\documentclass[preprint,12pt]{elsart}
\usepackage{graphicx}
\usepackage{amsfonts,amssymb,amsmath}
\usepackage{latexsym}
\usepackage{subfigure}
\usepackage{hyperref}
\usepackage{breakurl}
\usepackage[T1]{fontenc}
\usepackage[utf8]{inputenc}
\newenvironment{proof}{\par \noindent {\bf Proof}.\ }{\hfill$\Box$ \par \vspace{11pt}}
\newenvironment{proofof}[1]{\par \noindent {{\bf Proof} of #1.}}{\hfill$\Box$ \par \vspace{11pt}}

\begin{document}
\begin{frontmatter}
  \title{On triangles in $K_r$-minor free graphs\thanksref{ANR-EGOS}}
  \author[I3M]{Boris Albar},
  \ead{boris.albar@lirmm.fr}
  \author[LIRMM]{Daniel Gonçalves}
  \ead{daniel.goncalves@lirmm.fr}
  \address[I3M]{I3M \& LIRMM, CNRS and Univ. Montpellier 2, Place Eugène Bataillon, 34095
    Montpellier Cedex 5, France}
  \address[LIRMM]{LIRMM, CNRS and Univ. Montpellier 2, 161 rue Ada, 34095
    Montpellier Cedex 5, France}
  \thanks[ANR-EGOS]{This work was partially supported by the ANR grant
    EGOS 12 JS02 002 01}

  \begin{abstract}
    We study graphs where each edge adjacent to a vertex of small
    degree ($7$ and $9$, respectively) belongs to many triangles ($4$ and
    $5$, respectively) and show that these graphs contain a complete
    graph ($K_6$ and $K_7$, respectively) as a minor. The second case
    settles a problem of Nevo (Nevo, 2007). Morevover if each edge of
    a graph belongs to $6$ triangles then the graph contains a $K_8$-minor
    or contains $K_{2,2,2,2,2}$ as an induced subgraph. We then show
	applications of these structural properties to stress freeness and
	coloration of graphs. In particular, motivated by Hadwiger's conjecture,
	we prove that every $K_7$-minor free graph is $8$-colorable
	and every $K_8$-minor free graph is $10$-colorable.
  \end{abstract}
\end{frontmatter}

\section{Introduction}\label{sec:intro}

A minor of a graph $G$ is a graph obtained from $G$ by a succession of
edge deletions, edge contractions and vertex deletions.
All graphs we consider are simple, i.e. without loops or multiple edges.
The following theorem of Mader~\cite{mader1} bounds the number of edges in a
$K_r$-minor free graph.
\begin{thm}[Mader, 1968,~\cite{mader1}]
For $3 \leq r \leq 7$, any $K_r$-minor free graph $G$ on $n\ge r$
vertices has at most $(r-2)n - {{r-1}\choose{2}}$ edges.
\end{thm}
Note that since $|E(G)| = \frac{1}{2} \sum_{u \in  V(G)}\,\deg(u)$,
this theorem implies that every $K_r$-minor
free graph $G$, for $3 \leq r \leq 7$, is such that $\delta(G)\le
2r-5$, where $\delta(G)$ denotes the minimum degree of $G$.
This property will be of importance in the following. We are
interested in a sufficient condition for a graph to
admit a complete graph as a minor, dealing with the minimum number
of triangles each edge belongs to. Nevo~\cite{nevo1} already studied
this problem for small cliques. In the following, we assume
that every graph has at least one edge.
\begin{thm}[Nevo, 2007,~\cite{nevo1}]
For $3 \leq r \leq 5$, any $K_r$-minor free graph $G$ has an edge that
belongs to at most $r - 3$ triangles.
\label{th:nevoleq6}
\end{thm}
He also gave a weaker version for $K_6$-minor free graphs.
\begin{thm}[Nevo, 2007,~\cite{nevo1}]
Any $K_6$-minor free graph $G$ has an edge that belongs to at
most $r - 3$ triangles, or is a clique-sum over $K_r$, $r \leq 4$.
\label{th:nevok6}
\end{thm}

Nevo has conjectured that Theorem \ref{th:nevok6} can be extended to
the case of $K_7$-minor free graphs. We improve
Theorems~\ref{th:nevoleq6} and \ref{th:nevok6} in the following way.

\begin{thm}\label{th:krtri}
For $3 \leq r \leq 7$, any $K_r$-minor free graph $G$ has
an edge $uv$ such that $\deg(u) \leq 2r-5$ and $uv$ belongs to at
most $r-3$ triangles.
\end{thm}

In particular, this answers Nevo's conjecture about $K_7$-minor free graphs.
As pointed out by Nevo, Theorem~\ref{th:nevok6} cannot be further extended
to $K_8$-minor free graphs as such, since $K_{2,2,2,2,2}$
is a $K_8$-minor free graph whose every edge belongs to $6$
triangles. Actually, one can obtain $K_8$-minor free graphs whose every edge
belongs to $6$ triangles by gluing copies of $K_{2,2,2,2,2}$ on cliques of
any $K_8$-minor free graph. It is interesting to notice that $K_{2,2,2,2,2}$
appears in a Mader-like theorem for $K_8$-minor free graphs~\cite{jorg1}.
\begin{thm}[Jørgensen, 1994,~\cite{jorg1}]
Every graph on $n \geq 8$ vertices and at least $6n - 20$ edges either
has a $K_8$-minor, or is a $(K_{2,2,2,2,2}, 5)$-cockade (i.e. any
graph obtained from copies of $K_{2,2,2,2,2}$ by 5-clique sums).
\label{th:jorg}
\end{thm}

Although Theorem~\ref{th:krtri} cannot be extended to $K_8$-minor free
graphs, some similar conclusions can be reached by considering stronger
hypotheses. By increasing the minimum degree of the graph or
excluding $K_{2,2,2,2,2}$ as an induced subgraph, we have the following
three theorems.
\begin{thm}
Any $K_8$-minor free graph $G$ with $\delta(G)=11$ has an edge $uv$
such that $u$ has degree 11 and $uv$ belongs to at most $5$
triangles.
\label{th:k8triweak}
\end{thm}

\begin{thm}
Any $K_8$-minor free graph $G$ with $\delta(G)\ge 9$ has an edge that
belongs to at most $5$ triangles.
\label{th:k8-deg9}
\end{thm}

\begin{thm}
Any $K_8$-minor free graph $G$ with no $K_{2,2,2,2,2}$ as induced
subgraph has an edge that belongs to at most $5$ triangles.
\label{th:k8tri}
\end{thm}

We investigate applications of the previous results in the
rest of the paper.
In Section~\ref{sec:moytri}, we relax the hypothesis into a
more global condition on the overall number of triangles in
the graph. In particular, we prove that, for $3 \leq k \leq 7$
(resp. $k=8$), if a graph has $m\ge 1$ edges and at least
$\frac{k-3}{2}m$ triangles, then it has a $K_k$-minor (resp.
a $K_8$- or a $K_{2,2,2,2,2}$-minor).
In Section~\ref{sec:stress}, we show applications to stress freeness
of graphs, and settle some open problems of Nevo~\cite{nevo1}.
Finally, we show some applications to graph coloration
in Section~\ref{sec:double} and Section~\ref{sec:coloration}.
In the former section, we show an application to double-critical
$k$-chromatic graphs which settle a special case of a conjecture
of Kawarabayashi, Toft and Pedersen~\cite{kpt10}.
In the latter section, motivated by Hadwiger's conjecture,
we show that every $K_7$-minor free graph is $8$-colorable
and that every $K_8$-minor free graph is $10$-colorable.

\section{Proof of Theorem~\ref{th:krtri} for $r\le 6$ : A slight improvement of Nevo's theorem}\label{sec:proofk6}

First note that the cases $r=3$ or $4$ are trivial.  The case $r=5$ is
also quite immediate, but we need a few definitions to prove it. A
\emph{separation} of a graph $G$ is a pair $(A,B)$ of subsets of
$V(G)$ such that $A \cup B = V(G)$, $A\setminus B \neq \emptyset$,
$B\setminus A \neq \emptyset$, and no edge has one end in $A
\backslash B$ and the other in $B \backslash A$. The \emph{order} of a
separation is $|A \cap B|$.  A separation of order $k$ will be denoted
as a $k$-separation, and a separation of order at most $k$ as a $(\leq
k)$-separation. Given a vertex set $X\subseteq V(G)$ (eventually $X$
is a singleton) the sets $N(X)$ and $N[X]$ are respectively defined by
$\{y\in V(G)\setminus X \ |\ \exists x\in X\ {\rm s.t.}\ xy\in E(G)\}$
and $X\cup N(X)$.

Let us prove the case $r=5$. Consider any $K_5$-minor free graph
$G$. According to Wagner's characterization of $K_5$-minor free
graphs~\cite{w37}, $G$ is either the Wagner graph, a 4-connected
planar graph, or has a $(\le 3)$-separation $(A,B)$ such that $H=G[A]$
is either the Wagner graph or a 4-connected planar graph. If $G$ or
$H$ is the Wagner graph, as this graph has only degree 3 vertices and
no triangle, we are done. If $G$ (resp. $H$) is a 4-connected planar
graph, Euler's formula implies that  there is a vertex $v$
of degree at most 5 in $V(G)$ (resp. in $A \setminus B$). One can then observe
that, any edge around $v$ belongs to at most 2 triangles, as otherwise there would be
a separating triangle in $G$ (resp. $H$), contradicting its
4-connectivity.

Let us now focus on the case $r=6$ of Theorem~\ref{th:krtri}.
Consider by contradiction a $K_6$-minor free graph $G$ with at least
one edge, and such that every edge incident to a vertex of degree at
most $7$, belongs to at least $4$ triangles.  By Mader's theorem, we
have that $\delta(G)\le 7$.  We start by studying the properties of
$G[N(u)]$, for the vertices $u$ of degree at most 7. First, it is
clear that $G[N(u)]$ is $K_5$-minor free because otherwise there would
be a $K_6$-minor in $G$, contradicting the hypothesis.

\begin{lem}
$\delta(G)\ge 6$, and for any vertex $u$ of degree at most 7, $\delta(G[N(u)])\ge 4$.
\label{lem:mindeg6}
\end{lem}

\begin{proof}
For any vertex $u$ of degree at most 7, and any vertex of $v\in N(u)$
the edge $uv$ belongs to at least 4 triangles. The third vertex of
each triangle clearly belongs to $N(u)$ and is adjacent to $v$.  Thus
$v$ has degree at least $4$ in $G[N(u)]$.

Since for any vertex $u$ of degree at most $7$ we have
$\delta(G[N(u)])\ge 4$, $|N(u)|\ge 5$ (i.e. $\deg(u)\ge 5$).
Furthermore if there was a vertex $u$ of degree 5, as
$\delta(G[N(u)])\ge 4$, the graph $G[N(u)]$ would be isomorphic to
$K_5$, contradicting the fact that $G[N(u)]$ is $K_5$-minor free. Thus
$\delta(G)\ge 6$.
\end{proof}

As observed by Nevo (Proposition 3.3,~\cite{nevo1}), since
$|N(u)| \leq 7$, $\delta(G[N(u)]) \geq 4$ and $N(u)$ is $K_5$-minor
free, then $G[N(u)]$ is $4$-connected. Note that by Wagner's
characterization of $K_5$-minor free graphs, every $4$-connected
$K_5$-minor free is planar. Chen and Kanevsky~\cite{ck1} proved that
every $4$-connected graph can be obtained from $K_5$ and the
double-axle wheel $W_4^2$ by operations involving vertex splitting and
edge addition. Their result implies that the only two possibilities
for $G[N(u)]$ are the double-axle wheels on $4$ and $5$ vertices
depicted in Figure~\ref{fig:doublewheel}.  Note that theses two graphs
have $3|N(u)| - 6$ edges, and hence are maximal $K_5$-minor free (by
Mader's theorem).

\begin{figure}[h]
\centering
\subfigure{
\includegraphics[scale=1]{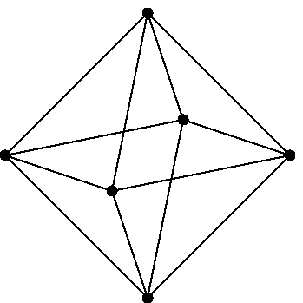}
}
\subfigure{
\includegraphics[scale=1]{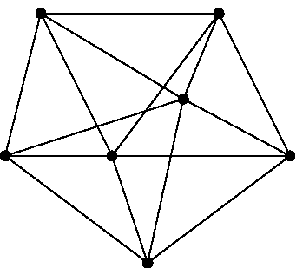}
}

\caption{The double-axle wheel on $4$ and $5$ vertices.}
\label{fig:doublewheel}
\end{figure}

We need the following lemmas on the neighborhood of the vertices with
small degree.

\begin{lem}
For any vertex $u$ of degree at most 7, every vertex $v\in N(u)$ has a
neighbor in $G\setminus N[u]$.
\label{lem:voisink6}
\end{lem}
\begin{proof}
Recall that $G[N(u)]$ is a double-axle wheel. Note that in a
double-axle wheel, every vertex has degree at most 5, and every edge
belongs to exactly 2 triangles.  Thus, every vertex of $N(u)$ has
degree at most 6 in $G[N[u]]$, and every edge of $G[N(u)]$ belongs to
exactly 3 triangles in $G[N[u]]$. This implies that any vertex $v\in
N(u)$ has either degree $>8$ in $G$, and thus at least 2 neighbors in
$G\setminus N[u]$, or that any of its incident edges $vw$ in $G[N(u)]$
is contained in a fourth triangle $vwx$, with $x\in G\setminus N[u]$.
\end{proof}

\begin{lem}
For any vertex $u$ of degree at most $7$, and any connected component
$C$ of $G\setminus N[u]$, the graph $G[N(C)]$ is a clique on at most 3
vertices.
\label{lem:voisin2k6}
\end{lem}
\begin{proof}
As $G[N(u)]$ has no clique on more than 3 vertices, let us show that
$N(C)$ does not contain two non-adjacent vertices , say $v_1$ and
$v_2$. There exists a path from $v_1$ to $v_2$ with inner vertices in
$C$.  Since $G[N(u)]$ is maximal $K_5$-minor free, this path together
with $G[N[u]]$ induces a $K_6$ minor in $G$, a contradiction.
\end{proof}

\begin{lem}
For any vertex $u$ of degree at most $7$, and any connected component
$C$ of $G\setminus N[u]$, there exists a vertex $u'\in C$ of degree at
most 7 in $G$.
\label{lem:existsk6}
\end{lem}
\begin{proof}
Suppose for contradiction that every vertex of $C$ has degree at least
8 in $G$.  Note that by definition, every vertex in $N(C)$ has a
neighbor in $C$.  Thus, as by Lemma~\ref{lem:voisin2k6} $G[N(C)]$ is a
clique on $k\le 3$ vertices, the vertices in $N(C)$ have degree at
least $k$ in $G[N[C]]$.  Thus the number of edges of $G[N[C]]$ is at
least
\[|E(G[N[C]])| \geq \frac{1}{2}(8|C| + k^2)  > 4(|C|+k) - 10\]
and by Mader's theorem, there is a $K_6$-minor in $G[N[C]]$, a
contradiction.
\end{proof}

Now choose a vertex $u$ of degree at most 7 and a connected component
$C$ of $G\setminus N[u]$, in such a way that $|C|$ is minimum.  By
Lemma~\ref{lem:existsk6}, $C$ has a vertex $v$ of degree at most 7.

Let $C_u$ be the connected component of $G\setminus N[v]$ that
contains $u$, and let $x\in N(v)\setminus N(C_u)$.  By
Lemma~\ref{lem:voisink6}, there is a connected component $C'$ of
$G\setminus N[v]$ such that $x\in N(C')$.

As $N[u]\subset N[C_u]$, it is clear that $G[C'\cup\{x,v\}]$ is a
connected subgraph of $G\setminus N[u]$. We thus have that
$C'\subsetneq C$ and thus that $|C'|<|C|$, contradicting the choice of
$u$ and $C$.  This concludes the proof of the case $r=6$ of
Theorem~\ref{th:krtri}.

\section{Proof of Theorem~\ref{th:krtri} for $r=7$ : the case of $5$ triangles}\label{sec:proofk7}

Consider by contradiction a $K_7$-minor free graph $G$ with at least
on edge, and such that every edge incident to a vertex of degree at
most $9$ belongs to at least $5$ triangles. By Mader's theorem,
$|E(G)| \leq 5|V(G)| - 15$, hence there are vertices $u$ such that
$\deg(u) \leq 9$.

We start by studying the properties of $G[N(u)]$, for any vertex $u$
of degree at most 9. First, it is clear that $G[N(u)]$ is $K_6$-minor
free because otherwise there would be a $K_7$-minor in $G$,
contradicting the hypothesis.

\begin{lem}
$\delta(G)\ge 7$, and for any vertex $u$ of degree at most 9, $\delta(G[N(u)])\ge 5$.
\label{lem:mindeg7}
\end{lem}

\begin{proof}
For any vertex $u$ of degree at most 9, and any vertex of $v\in N(u)$
the edge $uv$ belongs to at least 5 triangles. The third vertex of
each triangle clearly belongs to $N(u)$ and is adjacent to $v$.  Thus
$v$ has degree at least $5$ in $G[N(u)]$.

Since for any vertex $u$ of degree at most $9$ we have
$\delta(G[N(u)])\ge 5$, $|N(u)|\ge 6$ (i.e. $\deg(u)\ge 6$).
Furthermore if there was a vertex $u$ of degree 6, as
$\delta(G[N(u)])\ge 5$, the graph $G[N(u)]$ would be isomorphic to
$K_6$, contradicting the fact that $G[N(u)]$ is $K_6$-minor free. Thus
$\delta(G)\ge 7$.
\end{proof}

There is no appropriate theorem (contrarily to the previous case) to
generate all possible neighbourhoods of the small degree vertices.
Instead, we use a computer to generate all graphs with at most $9$
vertices and minimum degree at least 5. Then we refine (by computer)
our list of graphs, by removing the ones having a $K_6$-minor.  At the
end, we end up with a list of $22$ graphs. A difference with the
previous case is that not all the $22$ graphs are maximal $K_6$-minor
free graphs.  We deduce two of the following lemmas from the study of
$N(u)$ by computer~\cite{ag1}.

\begin{lem}
For any vertex $u$ of degree at most $9$, any connected component $C$
of $G\setminus N[u]$ is such that $|N(C)| = k \le 5$ and $|E(N(C))|\ge
{k \choose 2} -3$ (i.e. $G[N[C]]$ has at most 3 non-edges).
\label{lem:compk7}
\end{lem}

\begin{proof}
As any connected component $C$ could be contracted into a single
vertex, we prove the lemma by attaching a new vertex to all possible
combinations of $k$ vertices of $N[u]$ (as we know that $N(u)$ induces
one of the 22 graphs generated above), for any $k\le 6$, and check
when it induces a $K_7$-minor.
\end{proof}

This allows us to prove the following equivalent of
Lemma~\ref{lem:existsk6}.
\begin{lem}
For any vertex $u$ of degree at most $9$, any connected component $C$
of $G\setminus N[u]$ has a vertex $u'$ of degree at most 9 in $G$.
\label{lem:existsk7}
\end{lem}

\begin{proof}
Let $u$ be a vertex of $G$ of degree at most $9$ and let $C$ be a
connected component of $G\setminus N[u]$ which vertices have degree at
least 10 in $G$. Note that by definition every vertex of $N(C)$ has at
least one neighbor in $C$. Lemma~\ref{lem:compk7} implies that
$|N(C)|=k \leq 5$ and that $G[N(C)]$ has at most 3 non-edges. Thus,
contracting a conveniently choosen edge between $u$ and $N(C)$, one
obtains that $G[N(C)]$ has at most 1 non-edge.  After this
contraction, we have:
\begin{align*}
|E(N[C])| &\geq \frac{1}{2} \Big[ 10|C| + k(k-1) - 2 + k \Big] \\
&= 5|C| + \frac{k^2}{2} -1 > 5(|C|+k) - 15 .
\end{align*}
This contradicts the fact that $G[N[C]]$ is $K_7$-minor free, and thus
concludes the proof of the lemma.
\end{proof}

\begin{lem}
For any vertex $u$ of degree at most $9$, at most one vertex $v$ of
$N(u)$ is such that $N(v) \subseteq N[u]$.
\label{lem:numberk7}
\end {lem}

\begin{proof}
For every such vertex $v$, as $\deg(v)\le \deg(u)\le 9$, the edges
adjacent to $v$ with both ends in $N(u)$ belong to at least $5$
triangles in $G$ (i.e. belong to at least $4$ triangles in
$G[N(u)]$). We checked that for every graph in the list at most one
such vertex satisfies this condition.
\end{proof}

This allows us to prove the following lemma.
\begin{lem}
For any vertex $u$ of degree at most $9$ and any connected component
$C$ of $G\setminus N[u]$, there exists a connected component $C'$ of
$G\setminus N[u]$ such that $N(C')\setminus N(C) \neq \emptyset$.
\label{lem:numberk7bis}
\end {lem}
\begin{proof}
As $\deg(u)\ge 7$ (by Lemma~\ref{lem:mindeg7}) and $|N(C)|\le 5$ (by
Lemma~\ref{lem:compk7}), there are at least 2 vertices in
$N(u)\setminus N(C)$. By Lemma~\ref{lem:numberk7}, one of these 2
vertices has a neighbor $x$ out of $N[u]$.  Thus the component of
$G\setminus N[u]$ containing $x$ fulfills the requirements of the
lemma.
\end{proof}

Now choose a vertex $u$ of degree at most 9 and a connected component
$C$ of $G\setminus N[u]$, in such a way that $|C|$ is minimum.  By
Lemma~\ref{lem:existsk7}, $C$ has a vertex $v$ of degree at most 9.
Let $C_u$ be the connected component of $G\setminus N[v]$ that
contains $u$. By Lemma~\ref{lem:numberk7bis} there exists a connected
component $C'$ of $G\setminus N[v]$ such that $N(C')\setminus N(C_u)
\neq \emptyset$, and let $x\in N(C')\setminus N(C_u)$.  As
$N[u]\subset N[C_u]$, it is clear that $G[C'\cup\{x,v\}]$ is a
connected subgraph of $G\setminus N[u]$. We thus have that
$C'\subsetneq C$ and thus that $|C'|<|C|$, contradicting the choice of
$u$ and $C$.  This concludes the proof of case $r=7$ of
Theorem~\ref{th:krtri}

\section{Proof of Theorem~\ref{th:k8triweak}, \ref{th:k8-deg9} and \ref{th:k8tri} : the case of $6$ triangles}\label{sec:proofk8}

As in the previous sections, we will consider vertices of small degree
(and their neighborhoods) in $K_8$-minor free graphs.  We thus need the
following technical lemma that has been proven by computer~\cite{ag1}.

\begin{lem}
Every $K_7$-minors free graph $H$ distinct from $K_{2,2,2,2}$,
$K_{3,3,3}$ and $\overline{P_{10}}$ (the complement of the Petersen
graph), and such that $8\le |V(H)| \le 11$ and $\delta(H)\ge 6$,
verifies:
\begin{itemize}
\item $H$ is 5-connected.
\item $H$ has at most one vertex $v$ such that each of its
  incident edges belongs to 5 triangles.
\item For any subset $Y\subsetneq V(H)$ of size 7, the graph obtained
  from $H$ by adding two vertices $x$ and $y$ such that $N(x)=V(H)$
  and $N(y)=Y$, has a $K_8$-minor.
\end{itemize}
\label{lem:compk8}
\end{lem}
Note that the second property also holds for $K_{2,2,2,2}$,
$K_{3,3,3}$ and $\overline{P_{10}}$. Actually any edge of these 3
graphs belongs to less than 5 triangles.

By Theorem~\ref{th:jorg}, any $K_8$-minor free graph has minimum
degree at most 11.  Theorem~\ref{th:k8triweak} considers the case
where the minimum degree is exactly 11.  It will be used in
Section~\ref{sec:coloration} to color $K_8$-minor free graphs.

\begin{proofof}{Theorem~\ref{th:k8triweak}}
We prove this using the same technique as in
Section~\ref{sec:proofk7}.  Consider by contradiction a $K_8$-minor
free graph $G$ with $\delta(G)=11$, and such that every edge adjacent
to a degree 11 vertex belongs to at least $6$ triangles.  We start by
studying the properties of $G[N(u)]$, for any degree 11 vertex $u$.
First, it is clear that $G[N(u)]$ is $K_7$-minor free because
otherwise there would be a $K_8$-minor in $G$, contradicting the
hypothesis.

\begin{lem}
For any degree 11 vertex $u$, $\delta(G[N(u)])\ge 6$.
\label{lem:mindegk8}
\end{lem}
\begin{proof}
For any degree 11 vertex $u$ and any vertex of $v\in N(u)$,
the edge $uv$ belongs to at least 6 triangles. The third vertex of
each triangle clearly belongs to $N(u)$ and is adjacent to $v$.  Thus
$v$ has degree at least $6$ in $G[N(u)]$.
\end{proof}

%We can show a variation of Lemma~\ref{lem:existsk7}.
\begin{lem}
For any degree 11 vertex $u$, any connected component $C$
of $G\setminus N[u]$ has a vertex $u'$ of degree at most 11 in $G$.
\label{lem:existsk8}
\end{lem}
\begin{proof}
Let $u$ be a degree 11 vertex of $G$ and let $C$ be any connected
component of $G\setminus N[u]$ which vertices have degree at least 12
in $G$. Lemma~\ref{lem:compk8} implies that $G[N(u)]$ is 5-connected
and that $|N(C)|=k \leq 6$.  Thus the lemma holds by considering the
graph $G[N[u]\cup C]$ in the following Lemma~\ref{lem:existsk8bis}.
\end{proof}

\begin{lem}
A graph $H$ with a degree 11 vertex $u\in V(H)$ and such that:
\begin{itemize}
\item[(A)] $H[N(u)]$ is 5-connected,
\item[(B)] $\delta(H[N(u)])\ge 6$,
\item[(C)] the set $C=V(H)\setminus N[u]$ is non-empty, and all its vertices have degree at least 12, and
\item[(D)] the set $N(C) \subseteq N(u)$ has size $k\le 6$,
\end{itemize}
has a $K_8$-minor.
\label{lem:existsk8bis}
\end{lem}
\begin{proof}
Consider a minimal counter-example $H$, that is a $K_8$-minor free
graph $H$ fulfilling conditions (A), (B) (C) and (D), and minimizing
$|V(H)|$.  Note that by definition every vertex of $N(C) \subseteq
N(u)$ has at least one neighbor in $C$. Let us prove that actually
every vertex of $N(C)$ has at least 2 neighbors in $C$.  If $x\in
N(C)$ has only one neighbor $y$ in $C$, contract the edge $xy$ and
denote $G'$ the obtained graph. It is clear that $G'$ is $K_8$-minor
free, and fulfills conditions (A), (B) and (D). Moreover, $C\setminus
\{y\}$ is non-empty as it contains at least 6 vertices of $N(y)\cap C$
(as $\deg(y)\ge 12$ and $|N(C)|=k\le 6$), and every vertex of
$C\setminus \{y\}$ has degree at least 12 in $H'$ as none of these
vertices are adjacent to $x$ in $H$. So $G'$ also fulfills condition
(C), and this contradicts the minimality of $G$. Thus every vertex of
$N(C)$ has at least 2 neighbors in $C$.

One can easily see that every $(K_{2,2,2,2,2}, 5)$-cockade has at
least 10 degree 8 vertices. Thus the graph $H[N[C]]$, and any graph
obtained from $H[N[C]]$ by adding edges, cannot be a
$(K_{2,2,2,2,2}, 5)$-cockade as it has at most $6$ vertices of
degree $8$. Thus as $H[N[C]]$ has at least $\frac{1}{2}(12|C| + 2k)$
edges and as this is at least $6(|C|+k)-20$ for $k \leq 4$, by
Theorem~\ref{th:jorg} we have that $5\le k\le 6$.

Now suppose that $k = 5,6$. Let $v_1$ and $v_2$ be two vertices of
smallest degree in $H[N(C)]$. Denote $\delta_1$ and $\delta_2$ their
respective degree in $H[N(C)]$. Note that if $k=6$ then $\delta_1\ge
1$ as $v_1$ has at least $6$ neighbors in $N(u)$ and as there are only
$5$ vertices in $N(u)\setminus N(C)$. By contracting the edge $uv_1$,
we have $k - 1 -\delta_1$ additionnal edges in $H[N[C]]$. Moreover
since $H[N(u)]$ is $5$-connected and since $|N(C)| \leq 6$, for every
vertex $x\neq v_2$ of $N(C)$ we have $|N(C)\setminus\{x,v_2\}|=4$ and
thus the graph $H[N(u)]\setminus (N(C)\setminus\{x,v_2\})$ is
connected.  Thus, iteratively contracting all the edges between $v_2$
and $N(u)\setminus N(C)$ we add at least $k - 2 - \delta_2$ edges in
$H[N[C]]$ (as we have potentially already added the edge $v_1v_2$ in
the previous step). The number of edges in the obtained graph is at
least
\[\frac{1}{2}[(\delta_1 +2) + (\delta_2
  +2)(k-1)) + 12|C|] + (k - 1 -\delta_1) + (k - 2 -\delta_2)\]
which is more than $6(|C|+k)-20$ (as $k\le 6$ and as if $k=6$ then
$\delta_1 \ge 1$). Thus this graph has a $K_8$-minor, and so does $H$.
This completes the proof of the lemma.
\end{proof}

\begin{lem}
For any degree 11 vertex $u$ and any connected component
$C$ of $G\setminus N[u]$, there exists a connected component $C'$ of
$G\setminus N[u]$ such that $N(C')\setminus N(C) \neq \emptyset$.
\label{lem:numberk8bis}
\end {lem}
\begin{proof}
As $\deg(u)=11$ and $|N(C)|\le 6$ (by Lemma~\ref{lem:compk8}), there
are at least 5 vertices in $N(u)\setminus N(C)$.  As $\delta(G)=11$
one can easily derive from Lemma~\ref{lem:compk8} that one (actually,
at least 4) of these vertices has a neighbor $x$ out of $N[u]$.  Thus
the component of $G\setminus N[u]$ containing $x$ fulfills the
requirements of the lemma.
\end{proof}

Now choose a degree 11 vertex $u$ and a connected component $C$ of
$G\setminus N[u]$, in such a way that $|C|$ is minimum.  By
Lemma~\ref{lem:existsk8}, $C$ has a degree 11 vertex $v$.  Let $C_u$
be the connected component of $G\setminus N[v]$ that contains $u$. By
Lemma~\ref{lem:numberk8bis} there exists a connected component $C'$ of
$G\setminus N[v]$ such that $N(C')\setminus N(C_u) \neq \emptyset$,
and let $x\in N(C')\setminus N(C_u)$.  As $N[u]\subset N[C_u]$, it is
clear that $G[C'\cup\{x,v\}]$ is a connected subgraph of $G\setminus
N[u]$. We thus have that $C'\subsetneq C$ and thus that $|C'|<|C|$,
contradicting the choice of $u$ and $C$.  This concludes the proof of
Theorem~\ref{th:k8triweak}
\end{proofof}

let us now prove Theorem~\ref{th:k8tri}.  Given a counter-exemple $G$
of Theorem~\ref{th:k8tri}, note that adding a vertex $s$ to $G$,
adjacent to a single vertex of $G$, one obtains a counter-exemple of
the following theorem, thus Theorem~\ref{th:k8tri} is a corollary of
the following theorem.

\begin{thm}
Consider a connected $K_8$-minor free graph $G$ with a vertex $s$ of
degree at most 7 and such that $N[s] \subsetneq V(G)$.  If every edge
$e \in E(G) \setminus E(G[N[s]])$ belongs to at least 6 triangles, then
$G$ contains an induced $K_{2,2,2,2,2}$.
\label{th:k8triwiths}
\end{thm}
Note that as $K_{2,2,2,2,2}$ is maximal $K_8$-minor free, any
$K_8$-minor free graph $G$ containing a copy of $K_{2,2,2,2,2}=G[X]$,
for some vertex set $X\subseteq V(G)$, is such that any connected
component $C$ of $G\setminus X$ verifies that $N(C)$ induces a clique
in $G[X]$.
\begin{proof}
Consider a connected $K_8$-minor free graph $G$ with a vertex $s$ of
degree at most 7 such that $N[s] \subsetneq V(G)$, such that $G$ does not
contain an induced $K_{2,2,2,2,2}$, and such that every
edge $e \in E(G) \setminus E(G[N[s]])$ belongs to at least 6
triangles.  Assume also that $G$ minimizes the number of vertices.
This property implies that $G\setminus N[s]$ is connected. Indeed,
otherwise one could delete one of the connected components in
$G\setminus N[s]$ and obtain a smaller counter-example.  The graph $G$
is almost 8-connected as observed in the following lemma.
\begin{lem}
For any separation $(A,B)$ of $G$ (denote $S = A \cap B$), we have
either:
\begin{itemize}
\item $|S| \geq 8$, or
\item $s \notin S$ and $A\setminus B = \{s\}$ (i.e. $B = V(G) \setminus
  \{s\}$), or
\item $s \in S$ and $|S| \geq 6$.
\end{itemize}
\label{lem:connectk8withs}
\end{lem}
\begin{proof}
Suppose there exists a separation $(A,B)$ contradicting the lemma.
Note that $|S| < 8$ and let us assume that $s\in A$.

Consider first the case where $s\notin S = A\cap B$, that is the case
where $\{s\} \subsetneq A\setminus B$.  Assume that among all such
counter-examples, $(A,B)$ minimizes $|S|$.  In this case, if the
connected component of $A\setminus B$ containing $s$ has more vertices
then, contracting this component into $s$, one obtains a proper minor
$G'$ of $G$ such that $N[s]\subsetneq V(G')$ (as $B\setminus A \neq
\emptyset$) and such that every edge not in $E(N[s])$ belongs to $6$
triangles.  This would contradict the minimality of $G$, and we thus
assume the existence of a component $C_0=\{s\}$ in $G\setminus B$.  As
$\{s\} \subsetneq A\setminus B$, let $C_1\neq \{s\}$ be some connected
component of $G\setminus B$. Let also $C_2$ be some component of
$G\setminus A$.  Note that for any of these components $C_i$,
$N(C)\subsetneq S$. Otherwise one could contract (if needed) all the
component into a single vertex $s'$ and the graph induced by
$\{s\}\cup N[C_1]$ or by $\{s\}\cup N[C_2]$ (a proper minor of $G$)
would be a smaller counter-example. Note now that $S = N(s)\cup
N(C_i)$ for $i=1$ or 2.  Indeed, otherwise the separation $(N[s]\cup
N[C_i], V(G)\setminus (C_i\cup \{s\}))$ would be a counter-example
contradicting the minimality of $S$. Finally note that
$N(C_2)\not\subseteq N(C_1)$, as otherwise contracting $C_1$ into a
single vertex $s'$ and considering the graph induced by
$C_2\cup\{s'\}$ one would obtain a smaller counter-example. Thus there
exists a vertex $x\in N(s)\cap N(C_2)$ such that $x\notin
N(C_1)$. Contracting the edge $xs$ and contracting the whole component
$C_2$ into $x$, and considering the graph induced by $C_1\cup \{x\}$
one obtains a smaller counter-example (where $x$ plays the role of
$s$).
 
Consider now the case where $s\in S=A\cap B$ and note that $|S|< 6$.
Assume that among all the separations containing $s$, $(A,B)$
minimizes $|S|$. Note that every connected component $C$ of $G
\setminus S$ is such that $s\in N(C)$. Indeed, we have seen above that
otherwise $C$ would be such that $|N(C)|\ge 8$ (by considering the
separation $(V(G)\setminus C,N[C])$ and noting that $\{s\}\subsetneq
V(G)\setminus N[C]$), and this would contradict the fact that $|S|<
6$. This implies that every connected component $C$ of $G \setminus S$
is such that $N(C)=S$. Otherwise $(V(G)\setminus C, N[C])$ would be a
separation containing $s$ contradicting the minimality of $S$. We can
assume without loss of generality that $s$ has at most as many
neighbors in $B\setminus A$ than in $A\setminus B$. In particular,
since $\deg(s)\le 7$, $s$ has at most 3 neighbors in $B\setminus
A$. Note that $B \not\subseteq N[s]$ as otherwise $G \setminus
(B\setminus A)$ would be a smaller counter-example. Thus there is an
edge in $G[B\setminus N[s]$ that belongs to at least 6 triangles, and
thus $|B|\ge 9$ ($s$ and the 6 triangles). Thus contracting every
component of $A \setminus B$ on $s$, results in a proper minor $G'$
of $G$ such that $\deg(s) \leq 7$ (at most 4 in $S$ and 3 in
$B\setminus A$), such that $N[s]\subsetneq V(G')$ (as
$|V(G')|=|B|\ge 9$), and such that every edge not in $E(N[s])$
belongs to $6$ triangles, contradicting the minimality of $G$. This
concludes the proof of the lemma.
\end{proof}

By Theorem~\ref{th:jorg} $G$ has at most $6n-20$ edges, and thus there
are several vertices in $G$ with degree at most 11.  Let us prove that
there are such vertices out of $N[s]$.

\begin{lem}
There are at least 2 vertices in $V(G) \setminus N[s]$ with degree at
most 11.
\label{lem:2small_ink8}
\end{lem}
\begin{proof}
Assume for contradiction that every vertex of $V(G) \setminus N[s]$
but one, say $x$, has degree at least 12, and recall that such vertex
has degree at least 8. Note that every vertex $v\in N(s)$ has a
neighbor in $V(G) \setminus N[s]$, as otherwise $G\setminus v$ would
be a smaller counter-exemple. Thus every vertex $v\in N(s)$ has an
incident edge that belongs to at least 6 triangles (without using the
edge $sv$), which implies that $\deg(v)\ge 8$. This implies that the
number of edges in $G$ verifies :
\[ 12n-42 \ge 2|E(G)| = \sum_{v\in V(G)} \deg(v) \ge 8 + k + 8k + 12(n-k-2) \] 
where $k=\deg(s)$. This implies that $3k\ge 26$ which contradicts
the fact that $k=\deg(s) \le 7$. This concludes the proof of the
lemma.
\end{proof}

As for any vertex $u\in V(G)\setminus N[s]$ each of its incident edges
belongs to 6 triangles, the graph $G[N(u)]$ has minimum degree at
least 6. As $G$ does not contain $K_8$ as subgraph, this also implies
that $\deg(u)\ge 8$.  So there are at least two vertices in
$V(G)\setminus N[s]$ with degree between 8 and 11.  The next lemma
tells us more on the neighborhood of these small degree vertices.
\begin{lem}
For every vertex $u \in V(G) \setminus N[s]$ with degree at most 11 in
$G$, $G[N(u)]$ is isomorphic to $K_{2,2,2,2}$, $K_{3,3,3}$ or
$\overline{P_{10}}$.
\label{lem:NminDeg}
\end{lem}
\begin{proof}
Let $u$ be any vertex of $V(G) \setminus N[s]$ with degree at most 11
in $G$.  As observed earlier $8\le \deg(u)\le 11$ and
$\delta(G[N(u)])\ge 6$.  Assume for contradiction that $N(u)$, is not
isomorphic to $K_{2,2,2,2}$, $K_{3,3,3}$ or $\overline{P_{10}}$. Note
that $|N(u)\cap N(s)| \le 6$, as otherwise Lemma~\ref{lem:compk8}
would contradict the $K_8$-minor freeness of $G$.

By Lemma~\ref{lem:compk8} one of the (at least two) vertices in
$N(u)\setminus N(s)$, say $x$, has an incident edge in $G[N(u)]$ that
belongs to at most 5 triangles in $G[N[u]]$.  Thus the sixth triangle
containing this edge goes through a vertex $v$ of $V(G) \setminus
(N[u] \cup \{s\})$.

Lemma~\ref{lem:connectk8withs} implies that the connected component
$C$ of $v$ in $V(G) \setminus N[u]$ is such that $N(C)\ge 8$. The
graph obtained by contracting $C$ into a single vertex has a
$K_8$-minor (by Lemma~\ref{lem:compk8}), a contradiction.
\end{proof}

A $K_3$-minor rooted at $\{a, b, c\}$, or a $\{a, b, c\}$-minor, is a
$K_3$-minor in which you can contract edges incident to $a$, $b$ or
$c$, to obtain a $K_3$ with vertex set $\{a, b, c\}$.  For the rest of
the proof we need the following characterization of rooted $K_3$-minor.

\begin{thm}[D. R. Wood and S. Linusson, Lemma 5 of~\cite{wl1}]
For distinct vertices a, b, c in a graph G, either:
\begin{itemize}
\item $G$ contains an $\{a, b, c\}$-minor, or
\item for some vertex $v \in V(G)$ at most one of $a, b, c$ are in each component of $G \setminus v$.
\end{itemize}
\label{th:rootedtri}
\end{thm}

\begin{lem}
For every vertex $u \in V(G) \setminus N[s]$ with degree at most 11 in
$G$, the graph $G[N(u)]$ is not isomorphic to $K_{3,3,3}$.
\label{lem:noK333}
\end{lem}
\begin{proof}
Observe that adding two vertex disjoint edges or three edges of a
triangle in $K_{3,3,3}$ yields a $K_7$-minor.  Now assume for
contradiction that there exists some vertex $u\in V(G) \setminus N[s]$
such that $G[N(u)]$ is isomorphic to $K_{3,3,3}$.

As the set $N(u) \setminus N[s]$ is non-empty (it has size at least
$9-7$) and as every vertex $v$ in $N(u) \setminus N[s]$ has degree at
least $8$, and thus has a neighbor out of $N[u]$, $G \setminus N[u]$
has a connected component $C\neq \{s\}$.  By
Lemma~\ref{lem:connectk8withs} $|N(C)| \geq 8$.

If $G \setminus N[u]$ has another connected component $C'$ such that
$|N(C')| \geq 6$, one can create two vertex disjoint edges in
$K_{3,3,3}$ by contracting two vertex disjoint paths with non-adjacent
ends in $N(u)$, one living in each component. This would contradict
the $K_8$-minor freeness of $G$. Thus if there is a component $C'$,
we should have $C'=\{s\}$ and $\deg(s)\le 5$, as by
Lemma~\ref{lem:connectk8withs} a component $C'\neq \{s\}$ would be
such that $|N(C')| \geq 8$. In the following we consider the graph
$G' = G[N[u] \cup C]$ (which is $G$ or $G \setminus s$).

Let $\{a_1,a_2,a_3\}$, $\{b_1,b_2,b_3\}$ and $\{c_1,c_2,c_3\}$ be the
three disjoint stables of $N(u) = K_{3,3,3}$. Without loss of
generality we can assume that $\{a_1,a_2,a_3\}\subset N(C)$, and that
$a_1\notin N(s)$. As the edges of $N(u)$ incident to $a_1$ belong to
at least 6 triangles, $a_1$ has at least two neighbors in $G'
\setminus N[u]$. By Theorem~\ref{th:rootedtri} (applied to
$\{a_1,a_2,a_3\}$ in the graph $G''= G'\setminus
\{u,b_1,b_2,b_3,c_1,c_2,c_3\}$), there is a vertex $v \in V(G'')$ such
that at most one of $a_1, a_2, a_3$ are in each component of $G''
\setminus v$. Note that since $a_1$, $a_2$ and $a_3\in N(C)$, all the
sets $C\cup\{a_i,a_j\}$ induce a connected graph, and thus $v\neq
a_1$, $a_2$ or $a_3$. Equivalently we have that $v\in V(G') \setminus
N[u]$.  Hence $G'' \setminus \{v\}$ contains at least $3$ components
$C_1$, $C_2$ and $C_3$ with $a_i \in C_i$, for $1 \leq i \leq
3$. Since $a_1$ has at least two neighbors in $G' \setminus N[u]$, one
of them is distinct from $v$ and we can define $C'_1$ as a connected
component of $C_1\setminus \{a_1\}$. Note that by construction
$N(C'_1) \subset N(u)\cup\{v\}$.  Since $C'_1\neq\{s\}$ (as $a_1\notin
N(s)$) and as we might have $v=s$, Lemma~\ref{lem:connectk8withs}
implies that $N(C'_1)\ge 6$ (including $v$ and $a_1$). So $C'_1$ has
at least 4 neighbors in $\{b_1,b_2,b_3,c_1,c_2,c_3\}$ and there is a
path with interior vertices in $C'_1$ between two vertices $b_i$ and
$b_j$, or between two vertices $c_i$ and $c_j$. Furthermore, there is
a path with interior vertices in $C_2 \cup \{v\} \cup C_3$ between the
vertices $a_2$ and $a_3$. This contradicts the $K_8$-minor freeness of
$G$, and thus concludes the proof of the lemma.
\end{proof}

\begin{lem}
For every vertex $u \in V(G) \setminus N[s]$ with degree at most 11 in
$G$, the graph $G[N(u)]$ is not isomorphic to $K_{2,2,2,2}$.
\label{lem:noK2222}
\end{lem}
\begin{proof}
Assume for contradiction that there exists some vertex $u\in V(G)
\setminus N[s]$ such that $G[N(u)]$ is isomorphic to
$K_{2,2,2,2}$. One can check that adding two edges in $K_{2,2,2,2}$
creates a $K_7$-minor. Thus as $G$ is $K_8$-minor free it should not
be possible to add (by edge contractions) two new edges in $N(u)$.

\begin{claim}
A vertex $v\in V(G) \setminus N[u]$ has at most six neighbors in
$N(u)$.
\label{claim:k8-v6neighbors}
\end{claim}
\begin{proof}
If there was a vertex $v$ with 8 neighbors in $N(u)$, $N[u]\cup \{v\}$
would induce a $K_{2,2,2,2,2}$, a contradiction to the definition of
$G$. We thus assume for contradiction that there is a vertex $v$ with
exactly 7 neighbors in $N(u)$. Note that eventually $v=s$.  Let us
denote $x$ the only vertex in $N(u)\setminus N(v)$. Note that among
the 4 non-edges of $G[N(u)]$, only one cannot be created by
contracting an edge incident to $v$. So if there is a path whose ends
are non-adjacent in $N(u)$ and whose inner vertices belong to
$V(G)\setminus (N[u]\cup \{v\})$, then we have a $K_8$-minor, a
contradiction. There is clearly such path if $s\neq v$ and if $s$ has
5 neighbors in $N(u)$, we thus have that either $s=v$ or $s$ has at
most 4 neighbors in $N(u)$. Both cases imply that some edge $xy$
(incident to $x$) does not belong to $G[N[s]]$, and thus $xy$ belongs
to at least 6 triangles. As $xy$ belongs to only 5 triangles in
$G[N[u]]$, this implies the existence of a vertex $w\in V(G) \setminus
N[u]$ adjacent to $x$ such that $w\neq s, v$. Let $C$ be the connected
component of $w$ in $G\setminus (N[u]\cup \{v\})$. As $C\neq \{s\}$,
Lemma~\ref{lem:connectk8withs} implies that $N(C)$ has size at least
6. Thus $C$ has at least 5 neighbors in $N(u)$ and one can link two
non-adjacent vertices of $N(u)$ by a path going through $C$, a
contradiction.
\end{proof}

By Lemma~\ref{lem:2small_ink8} there exists another vertex $u'\in V(G)
\setminus N[s]$ such that $\deg(u')\le 11$. By Lemma~\ref{lem:NminDeg}
and Lemma~\ref{lem:noK333}, $G[N(u')]$ is isomorphic to $K_{2,2,2,2}$
or $\overline{P_{10}}$.

\begin{claim}
The vertices $u$ and $u'$ are non-adjacent.
\end{claim}
\begin{proof}
We assume for contradiction that $u$ and $u'$ are adjacent and we
first consider the case where $G[N(u')]$ is isomorphic to
$K_{2,2,2,2}$. In this case, as $u'$ has allready 7 neighbors in
$N[u]$, $u'$ has a exactly one neighbor $v$ in $G\setminus N[u]$.  As
$v$ has 7 neighbors in $N(u')$, we have that $|N(u)\cap N(v)|\ge 7$, a
contradiction to Claim~\ref{claim:k8-v6neighbors}.

If $G[N(u')]$ is isomorphic to $\overline{P_{10}}$, this implies that
$G[N(u)\cap N(u')]$ is isomorphic to $\overline{C_6}$ (the complement
of the 6-cycle). This is not compatible with $G[N(u)]$ being
isomorphic to $K_{2,2,2,2}$, as this in turn implies that $G[N(u)\cap
N(v)]$ is isomorphic to $K_{2,2,2}$.
\end{proof}

As by Lemma~\ref{lem:connectk8withs} there is no $(\le 5)$-separator
$(A,B)$ with $u\in A\setminus B$ and $u'\in B\setminus A$, Menger's
Theorem implies the existence of 6 vertex disjoint paths between $u$
and $u'$.  These paths induces $6$ disjoint paths $P_1 \ldots P_6$
between $N(u)$ and $N(u')$. Note that every vertex in $N(u) \cap
N(u')$ can be seen as a path of length $0$.

Therefore, since $N(u)$ is isomorphic to $K_{2,2,2,2}$, there are two
non-edges $a_1a_2$ and $a_3a_4$ of $G[N(u)]$ such that each $a_i$ is
the end of the path $P_i$. We denote by $b_i$, $1 \leq i \leq 4$ the
end in $N(u')$ of the path $P_i$.  Note that if $a_i \in N(u) \cap
N(u')$ then $a_i = b_i$.  Moreover we can suppose that the choice of
$a_1a_2$ and $a_3a_4$ maximizes the size of $\{a_1,a_2,a_3,a_4\} \cap
N(u')$. Since $N(u)$ is isomorphic to $K_{2,2,2,2}$ and since $|N(u)
\cap N(u')| \leq 6$ (by Claim~\ref{claim:k8-v6neighbors}), there are
at most two vertices in $N(u) \cap N(u')$ distinct from $a_1,a_2,a_3$,
and $a_4$. Let $X= (N(u) \cap N(u')) \setminus \{a_1,a_2,a_3,a_4\}$.

Since both $K_{2,2,2,2}$ and $\overline{P_{10}}$ are $6$-connected
then $N[u']$ is $7$-connected and so $G[N[u'] \setminus X]$ is
$5$-connected.  Moreover $G[N[u'] \setminus X]$ has too many edges to
be planar. Indeed, it has $9-|X|$ vertices and at least $32-7|X|$
edges, which is more than $3(9-|X|)-6$ for $0 \le |X| \le 2$. We now
need the following theorem of Robertson and Seymour about vertex
disjoint pairs of paths.

\begin{thm}[Robertson and Seymour~\cite{rs1}]
\label{th:disjointpath}
Let $v_1 , \ldots, v_k$ be distinct vertices of a graph $H$. Then either
\begin{itemize}
\item[(i)] there are disjoint paths of $H$ with ends $p_1$ $p_2$ and
  $q_1$ $q_2$ respectively, so that $p_1$, $q_1$, $p_2$, $q_2$ occur
  in the sequence $v_1, \ldots, v_k$ in order, or
\item[(ii)] there is a $(\le 3)$-separation $(A,B)$ of $H$ with $v_1,
  \ldots, v_k \in A$ and $|B \setminus A| \geq 2$, or
\item[(iii)] $H$ can be drawn in a disc with $v_1 , \ldots, v_k$ on
  the boundary in order.
\end{itemize}
\end{thm}

Applying this theorem to the graph $G[N[u'] \setminus X]$ with
$(v_1,\ldots v_k) = (b_1,b_3,b_2,b_4)$ one obtains that there are two
vertex disjoint paths in $N[u'] \setminus X$, a path $P_{1,2}$ between
$b_1$ and $b_2$, and a path $P_{3,4}$ between $b_3$ and $b_4$.  Theses
paths are disjoint from $N[u]$ by construction, except possibly at
their ends.  Finally, since the paths $P_i$, for $1 \leq i \leq 4$,
constructed above are disjoint from $N[u]$ and from $N[u'] \setminus
X$, except at their ends, there exists two disjoint paths respectively
linking $a_1$ with $a_2$ (through $P_1$, $P_{1,2}$ and $P_2$), and
$a_3$ with $a_4$ (through $P_3$, $P_{3,4}$ and $P_4$). This
contradicts the $K_8$-minor freeness of $G$ and thus concludes the
proof of the lemma.
\end{proof}

By Lemma~\ref{lem:2small_ink8} there exists at least two vertices
$u$ and $u'\in V(G) \setminus N[s]$ with degree at most $11$. By
Lemma~\ref{lem:NminDeg}, Lemma~\ref{lem:noK333}, and
Lemma~\ref{lem:noK2222}, both $G[N(u)]$ and $G[N(u')]$ are isomorphic
to $\overline{P_{10}}$. The two graphs induced by $N[u]$ and $N[u']$
are close to a $K_8$-minor as observed in the following claim.
\begin{claim}
\label{cm:2edgeP10}
In $\overline{P_{10}}$, adding two edges $ab$ $cd$, such that $ab$,
$bc$ and $cd \notin E(\overline{P_{10}})$, creates a $K_7$-minor.
Furthermore adding three edges $e_1$ $e_2$ and $e_3$, such that $e_1
\cap e_2 \cap e_3 =\emptyset$ in $\overline{P_{10}}$, creates a
$K_7$-minor.
\end{claim}
\begin{proof}
One can easily check the accuracy of the first statement, by noting
that adding any such pair of edges $ab$ and $cd$, yields the same
graph, and by noting that adding the edges $u_1u_2$ and $u_3u_4$ in
$\overline{P_{10}}$ (notations come from Figure~\ref{fig:P10}) the
partition $\{\{0,2\},\{1\},\{3\},\{4\},\{5\}, \{6,7\}, \{8,9\}\}$
induces a $K_7$-minor.

For the second statement, we can assume that the three added edges are
such that they pairwise do not correspond to the first statement.
Without loss of generality, assume that one of the three edges is
$u_0u_5$, and note that the other added edges are distinct from
$u_1u_2$, $u_1u_6$, $u_3u_4$, $u_4u_9$, $u_2u_7$, $u_7u_9$, $u_3u_8$
and $u_6u_8$. Consider the case where one of the other added edges is
incident to $u_0u_5$. By symmetry one can assume that this edge is
$u_0u_1$, but this implies that the third added edge is distinct from
$u_0u_4$ (as the three edges would intersect), and from $u_3u_4$,
$u_6u_9$, $u_5u_7$ and $u_5u_8$ (by the first statement). There is
thus no remaining candidate for the third edge.
This implies that it is sufficient to consider the case where the
edges $u_0u_5$, $u_2u_3$ and $u_6u_9$ are added in
$\overline{P_{10}}$. In this case the partition $\{\{1\},\{ 4\},\{
7\},\{ 8\},\{ 0,5\},\{ 6,9\},\{ 2,3\} \}$ induces a $K_7$-minor.
\end{proof}
\begin{figure}[h]
\centering
\includegraphics{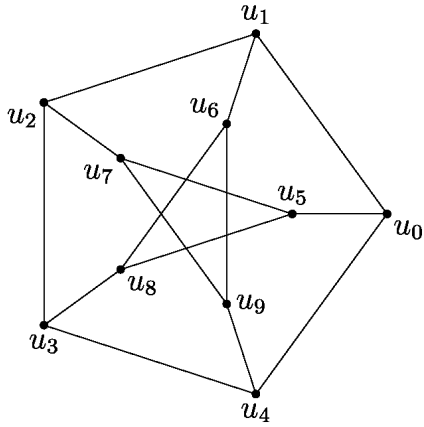}
\caption{The Petersen graph $P_{10}$.}
\label{fig:P10}
\end{figure}

Let us list the induced subgraphs of $\overline{P_{10}}$ of size 6.
\begin{claim}
\label{cm:k8-P10-6subgraphs}
There are exactly 6 distinct induced subgraphs of size 6 in
$\overline{P_{10}}$, including $K_{2,2,2}$. The complements of these
graphs are represented in Figure~\ref{fig:P10-6subgraphs}.
Furthermore note that every induced subgraphs of $\overline{P_{10}}$
of size at least 7, has a subgraph of size 6 distinct from
$K_{2,2,2}$.
\end{claim}
We do not prove the claim here as one can easily check its accuracy.
\begin{figure}[h]
\centering
\includegraphics[width=390px]{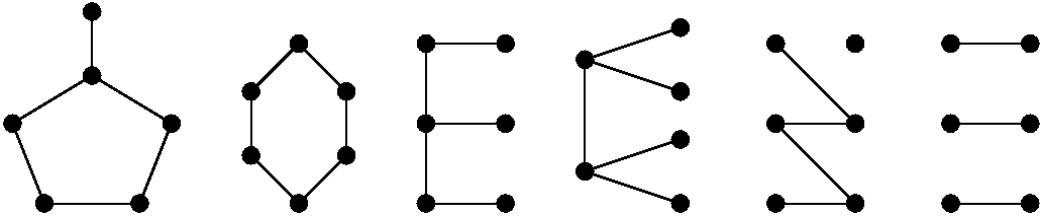}
\caption{The complements of the subgraphs of $\overline{P_{10}}$ of
  size 6 (i.e. the subgraphs of $P_{10}$ of size 6).}
\label{fig:P10-6subgraphs}
\end{figure}

\begin{lem}
The vertices of $N(u)\setminus N(s)$ (resp. of $N(u')\setminus N(s)$)
have degree at least 12. Thus in particular, $u$ and $u'$ are
non-adjacent.
\label{lem:k8-P10-deg12}
\end{lem}
\begin{proof}
We assume for contradiction that $u$ has a neighbor $v$ of degree at
most 11.  By Lemma~\ref{lem:NminDeg}, Lemma~\ref{lem:noK2222}, and
Lemma~\ref{lem:noK333}, the graph $G[N(v)]$ is isomorphic to
$\overline{P_{10}}$.

Assume $v=u_0$ in Figure~\ref{fig:P10}. Since $N(u_0) \supset
\{u,u_2,u_3,u_6,u_7,u_8,u_9\}$, the adjacencies in
$G[\{u,u_2,u_3,u_6,u_7,u_8,u_9,\}]$ allow us to denote $u$ by $u'_0$,
and denote $u'_1$, $u'_4$ and $u'_5$ the vertices in $N(u_0)\setminus
N[u]$, in such a way that these indices again correspond to
Figure~\ref{fig:P10}. It is now easy to see that contracting one edge
in each of the paths $(u_2,u'_4,u_7)$ and $(u_6,u'_5,u_9)$ creates the
edges $u_2u_7$ and $u_6u_9$ in $G[N[u]]$ and thus yields a $K_8$-minor
(by Claim~\ref{cm:2edgeP10} as $u_7u_9$ is a non-edge of
$\overline{P_{10}}$), a contradiction.
\end{proof}

The vertices $u$ and $u'$ are non-adjacent, however they can share
neighbors.  Let us prove that they cannot share more than 7 neighbors.
\begin{lem}
$|N(u) \cap N(u')|\le 7$.
\label{lem:k8-P10->=7common}
\end{lem}
\begin{proof}
Assume for contradiction that $|N(u) \cap N(u')|\ge 8$, that is
equivalently that $|N(u) \setminus N(u')|\le 2$ and $|N(u') \setminus
N(u)|\le 2$.  Note that as $\deg(s)\le 7$ the set $(N(u) \cap N(u'))
\setminus N(s)$ is non-empty, and denote $x$ one of its vertices.  By
Lemma~\ref{lem:k8-P10-deg12}, this vertex $x$ as degree at least $12$.
As it has exactly 6 neighbors in $N(u)$, at most 2 neighbors in $N(u')
\setminus N(u)$, and as it is adjacent to both $u$ and $u'$, $x$ has
at least two neighbors in $V(G) \setminus (N[u] \cup N[u'])$. Thus
there exists a component $C\neq\{s\}$ in $G \setminus (N[u] \cup
N[u'])$. As $C\neq\{s\}$ and $N(C)\subseteq N(u)\cup N(u')$,
Lemma~\ref{lem:connectk8withs} implies that $|N(C)|\ge 8$. Therefore,
as $|N(u') \setminus N(u)|\le 2$, $|N(C) \cap N(u)| \ge 6$ and there
exist a path $P$ with inner vertices in $C$ and with non-adjacent ends
in $N(u)$ (by Claim~\ref{cm:k8-P10-6subgraphs}). Let us denote $x$ and
$y$ the ends of $P$. As $|N(u) \cap N(u')|\ge 8$ and by
Claim~\ref{cm:k8-P10-6subgraphs}, there exists a vertex $z\in N(u)
\cap N(u')$ such that $z\neq x$ or $y$, and such that contracting the
edge $zu'$ creates at least two edges in $N(u)$. As these three added
edges ($xy$ and the edges adjacent to $z$) do not intersect,
Claim~\ref{cm:2edgeP10} implies that there is a $K_8$-minor, a
contradiction.
\end{proof}

As by Lemma~\ref{lem:connectk8withs} there is no $(\le 5)$-separator
$(A,B)$ with $u\in A\setminus B$ and $u'\in B\setminus A$, Menger's
Theorem implies the existence of 6 vertex disjoint paths $P_1 \ldots
P_6$ between $u$ and $u'$.
By minimizing the total length of these paths we can assume that
each vertex in $N(u)\cap N(u')$ corresponds to one of these paths,
and that any of these paths intersect $N(u)$ (resp. $N(u')$) in
only one vertex. Contracting the inner edges (those non-incident to
$u$ or $u'$) of these paths, and considering the graph induced by
$N[u]\cup N[u']$ one obtains a graph $H$ such that:
\begin{itemize}
\item $u$ and $u'$ are nonadjacent and $|N_H(u)\cap N_H(u')| = 6$ or
  $7$.
\item $\deg_H(u)=10$, and $H[N(u)]$ contains $\overline{P_{10}}$ as a
  subgraph.
\item $\deg_H(u')=10$, and $H[N(u')]$ contains $\overline{P_{10}}$ as
  a subgraph.
\end{itemize}

If the graph induced by $N_H(u)\cap N_H(u')$ is isomorphic to
$K_{2,2,2}$, then one can assume without loss of generality that
$N(u)=\{u_0,\ldots,u_9\}$ and that $N(u')=\{u_0,u'_1, u_2,u_3,
u'_4,u_5,u_6, u'_7,u'_8,u_9\}$, where the indices correspond to
Figure~\ref{fig:P10}. Now observe that contracting the edge $u_0u'$,
the path $(u_6,u'_7,u'_8)$, and the path $(u_2,u'_4,u'_1)$,
respectively create the edges $u_0u_5$, $u_6u_9$, and $u_2u_3$.  This
implies by Claim~\ref{cm:2edgeP10} that $N[u]$ contains a $K_8$-minor,
a contradiction. We can thus assume by
Claim~\ref{cm:k8-P10-6subgraphs} that the complement of $N_H(u)\cap
N_H(u')$ contains a path $(a,b,c,d)$.  As $\overline{P_{10}}$ is
6-connected, the graph induced by $\{a,b\}\cup (N_H(u')\setminus
N(u))$ is connected, and thus contains a path from $a$ to $b$.  By
Claim~\ref{cm:2edgeP10}, this path with the path $(c,u',d)$, imply
that $H$ (which is a minor of $G$) contains a $K_8$-minor, a
contradiction. Thus there is no counter-example $G$, and this
concludes the proof of the theorem.
\end{proof}

The proof Theorem~\ref{th:k8-deg9} is very similar. To do this one can
prove the following variant of Theorem~\ref{th:k8triwiths}.
\begin{thm}
Consider a connected $K_8$-minor free graph $G$ with a vertex $s$ of
degree at most 7, such that $N[s] \subsetneq V(G)$ and such that
$\min_{v\in V(G)\setminus N[s]}\ge 9$. Then $G$ has an edge
$e \in E(G) \setminus E(G[N[s]])$ that belongs to at most 5 triangles.
\end{thm}
The proof of this theorem is as the proof of
Theorem~\ref{th:k8triwiths}, except that one does not need to consider
the case where some vertex $u$ is such that $N(u)$ induces a
$K_{2,2,2,2}$.

\section{Global density of triangles}\label{sec:moytri}

In this section, we investigate the relation between the number of
triangles and the number of edge of a graph. Denotes by $\rho =
\frac{t}{m}$ the ratio between the number of triangles $t$ and the
number of edges $m$ of a graph $G$.  For each $k$, what is the minimum
number $f(k)$ such that for all graph $G$ with $\rho \ge f(k)$, $G$
contains a $K_k$ minor ?

It is easy to notice that 2-trees on $n \ge 2$ vertices have exactly
$1 + 2(n-2)$ edges and $n-2$ triangles.  Furthermore, for $k \ge 3$
one can notice that $k$-trees on $n \ge k$ vertices have exactly
$\frac{k(k-1)}{2} + k(n-k)$ edges and $\frac{k(k-1)(k-2)}{6} +
(n-k)\frac{k(k-1)}{2}$ triangles. Thus any $k$-tree, for $k \ge 2$,
verifies
\[ t = \frac{k-1}{2} m -\frac{1}{2}{{k+1}\choose{3}} .\]
Since $k$-trees are $K_{k+2}$-minor free, for all $k \ge 4$ there
exists $K_k$-minor free graphs with $\frac{k-3}{2} m
-\frac{1}{2}{{k-1}\choose{3}}$ triangles.

We deduce that for all $k \ge 4$, $f(k) \geq \frac{k-3}{2}$.  Indeed
for every $\epsilon > 0$, there exists a number $m$ and a $K_k$-minor
free graph with $m$ edges such that $\frac{k-3}{2} - \epsilon \leq
\rho < \frac{k-3}{2}$.  In fact, for $4 \le k \le 7$, the following
theorem proves that this lower bound is best possible, so we have
$f(k) = \frac{k-3}{2}$.

\begin{thm}
For $4 \le k \le 7$ (resp. $k = 8$), every graph with $m \ge 1$ edges
and $t \ge m(k-3)/2$ triangles has a $K_k$-minor (resp. a $K_8$- or a
$K_{2,2,2,2,2}$-minor).
\label{thm:global}
\end{thm}

\begin{proof}
Consider by contradiction, a non-trivial $K_k$-minor free
(resp. $K_8$- and $K_{2,2,2,2,2}$-minor free) graph $G$ with $t
\ge m(k-3)/2$ triangles.  Among the possible graphs $G$, consider one
that minimizes $m$ (given that $m \ge 1$).

Given any edge $uv \in E(G)$ let $H_{uv} = G[N(u) \cap N(v)]$ and
denote $n'$ and $m'$ its number of vertices and edges respectively.
Contracting $uv$ yields a proper minor of $G$, with exactly $1 + n'$
edges less, and with at most $n' + m'$ triangles less. Thus by
minimality of $G$, for every edge $uv$
\[ n' + m'  >  \frac{k-3}{2} (1 + n') \]
which implies that
\[ m'  >  \frac{k-3}{2} + \frac{k-5}{2} n' .\]
On the other hand we have that $\frac{n'(n' - 1)}{2} \ge m'$, and this
implies that $n'$ should verify $(n' + 1)(n' + 3 - k) > 0$, that is
that $n' \ge k - 2$. In other words, every edge $uv$ of $G$ belongs to
at least $k-2$ triangles. By Theorems~\ref{th:krtri},
(resp. Theorem~\ref{th:k8tri}), this contradicts the $K_k$-minor
freeness (resp. $K_8$- and $K_{2,2,2,2,2}$-minor freeness) of $G$.
\end{proof}

\section{Application to stress freeness of graphs}\label{sec:stress}

The motivation of this application is a problem that arises from the
study of tension and compression forces applied on frameworks in the
Euclidian space $\mathbb{R}^d$. A $d$-framework is a graph $G=(V,E)$
and an embedding $\rho$ of $G$ in $\mathbb{R}^d$. The reader should
think of a framework as an actual physical system where edges are
either straight bars or cables and vertices are articulated joints.  A
\textit{stress} on a framework $(G,\rho)$ is a function $\omega:
E(G)\, \rightarrow\; \mathbb{R}$ such that $\forall v \in V$,
\[\underset{\{u,v\}\in E}{\sum}\:\omega(\{u,v\})(\rho(v) - \rho(u)) = 0 .\]
Stress corresponds to some notion of equilibrium for the associated
physical system.  Each vertex is affected by tension and compression
forces created by the bars and cables.  $\omega(\{u,v\})$ can be
thought of as the magnitude of such force per unit length, with
$\omega(\{u,v\}) < 0$ for a cable tension and $\omega(\{u,v\}) > 0$
for a bar compression.  A stress is a state of the system where these
forces cancel each other at every vertex.  We can see that every
framework admits a \textit{trivial} stress where $\omega$ is
identically zero.  A $d$-framework admitting only the trivial stress
is called \textit{$d$-stress free}.

To make this notion independent of the embedding of $G$, the following
was introduced.  A graph $G$ is \textit{generically $d$-stress free}
if the set of all $d$-stress free embeddings of $G$ in $\mathbb{R}^d$
is open and dense in the set of all its embeddings (i.e. every
stressed embedding of $G$ is arbitrary close to a stress free
embedding).

This notion has been first used on graphs coming from $1$-skeletons of
$3$-dimensional polytopes
\cite{cauchy-13,maxwell-64,cw-93,whiteley-84}, which are planar by
Steiniz's theorem.  Gluck generalized the results on $3$-dimensional
polytopes to the whole class of planar graphs.
\begin{thm}[Gluck, 1975,~\cite{gluck-75}]
Planar graphs are generically $3$-stress free.
\label{th:gluckstress}
\end{thm}
Nevo proved that we can generalize Theorem~\ref{th:gluckstress} for
$K_5$-minor free graphs, and extended the result as follows.

\begin{thm}[Nevo, 2007,~\cite{nevo1}]
For $2 \leq r \leq 6$, every $K_r$-minor free graph is
generically $(r-2)$-stress free.
\label{th:nevostress}
\end{thm}

He conjectured this to hold also for $r = 7$ and noticed that the
graph $K_{2,2,2,2,2}$ is an obstruction for the case $r=8$. Indeed,
$K_{2,2,2,2,2}$ is $K_8$-minor free and has too many edges to be
generically $6$-stress free (a generically $\ell$-stress free graph
has at most $\ell n - {\ell+1 \choose 2}$ edges~\cite{nevo1}).  We
answer positively to Nevo's conjecture and we give a variant for the
generically $6$-stress freeness.

\begin{thm}
Every $K_7$-minor free graph (resp. $K_8$- and $K_{2,2,2,2,2}$-minor
free graph) is generically $5$-stress free (resp. $6$-stress free).
\label{th:genstressk78}
\end{thm}

The following result of Whiteley~\cite{whiteley-89} is used to derive
Theorem~\ref{th:genstressk78}.
\begin{thm}[Whiteley, 1989,~\cite{whiteley-89}]
Let $G'$ be obtained from a graph $G$ by contracting an edge $\{u,v\}$.
If $u$, $v$ have at most $d - 1$ common neighbors and $G'$ is generically
$d$-stress free, then $G$ is generically $d$-stress free.
\label{th:whitcontract}
\end{thm}

Now, we prove Theorem~\ref{th:genstressk78}.
\begin{proof}
Assume that $G$ is a $K_7$-minor free graph (resp. a $K_8$- and
$K_{2,2,2,2,2}$-minor free graph). Without loss of generality, we can
also assume that $G$ is connected.  Now, contract edges belonging to
at most $4$ (resp. 5) triangles as long as it is possible and we
denotes by $G'$ the graph obtained. Note that by construction, every
edge of $G'$ belongs to 5 (resp. 6) triangles. Note also that $G'$ is
a minor of $G$, and is thus $K_7$-minor free (resp.  $K_8$- and
$K_{2,2,2,2,2}$-minor free). Theorem~\ref{th:krtri}
(resp. Theorem~\ref{th:k8tri}) thus implies that $G'$ is the trivial
graph without any edge and with one vertex. This graph is trivially
generically $5$-stress free (resp. $6$-stress free), and so by
Theorem~\ref{th:whitcontract}, $G$ also is generically $5$-stress free
(resp. $6$-stress free).
\end{proof}

We denote by $\mu(G)$ the Colin de Verdière parameter of a graph $G$.  A
result of Colin de Verdière~\cite{cdv1} is that a graph $G$ is planar
if and only if $\mu(G) \leq 3$. Lov\'asz and Schrijver~\cite{ls1}
proved that $G$ is linklessy embeddable if and only if $\mu(G) \leq
4$.  Nevo conjectured that the following holds.

\begin{conj}[Nevo, 2007,~\cite{nevo1}]
Let $G$ be a graph and let $k$ be a positive integer. If $\mu(G) \leq k$
then $G$ is generically $k$-stress free.
\label{conj:munevo}
\end{conj}

This conjecture holds for the cases $k = 5$ and $k = 6$ as a
consequence of Theorem~\ref{th:genstressk78}.

\begin{cor}
If $G$ is a graph such that $\mu(G) \leq 5$ (resp. $\mu(G) \leq 6$) then
$G$ is generically $5$-stress free (resp. $6$-stress free).
\end{cor}

\begin{proof}
Note that $\mu(K_r) = r - 1$ and that if the complement of an
$n$-vertex graph $G$ is a linear forest, then $\mu(G) \geq n -
3$~\cite{lsv1}. So we have that $\mu(K_7) = 6$, $\mu(K_8) = 7$,
and $\mu(K_{2,2,2,2,2}) \geq 7$.

As the parameter $\mu$ is minor-monotone~\cite{cdv1}, the graph $K_7$
(resp. $K_8$ and $K_{2,2,2,2,2}$) is an excluded minor for the class
of graphs defined by $\mu(G) \leq 5$ (resp. $\mu(G) \leq 6$). Hence by
Theorem~\ref{th:genstressk78}, these graphs are generically $5$-stress
free (resp. $6$-stress free).
\end{proof}

\section{Application to double-critical $k$-chromatic graphs}\label{sec:double}

A connected $k$-chromatic graph is said to be double-critical is for
all edge $uv$ of $G$, $\chi(G \setminus\{u,v\}) = \chi(G) - 2$. It is
clear that the clique $K_k$ is such a graph.  The following
conjecture, known has the Double-Critical Graph Conjecture, due to
Erdős and Lovász states that they are the only ones.

\begin{conj}[Erdős and Lovász, 1968,~\cite{e68}]
If $G$ is a double-critical $k$-chromatic graph, then $G$ is
isomorphic to $K_k$.
\label{conj:dcgraph}
\end{conj}

This conjecture has been proved for $k \leq 5$ but remains open for $k
\geq 6$.  Kawarabayashi, Pedersen and Toft have formulated a relaxed
version of both Conjecture~\ref{conj:dcgraph} and the Hadwiger's
conjecture, called the Double-Critical Hadwiger Conjecture.

\begin{conj}[Kawarabayashi, Pedersen, and Toft, 2010,~\cite{kpt10}]
If $G$ is a double-critical $k$-chromatic graph, then $G$ contains a
$K_k$-minor.
\label{conj:dchadwiger}
\end{conj}

The same authors proved this conjecture for $k \leq 7$~\cite{kpt10},
but the case $k = 8$ is left as an open problem. Pedersen proved that
every $8$-chromatic double-critical contains $K_8^-$ as a
minor~\cite{ped11}. Below we prove that the conjecture also holds for
$k=8$.

The following proposition lists some interesting properties about
$k$-chromatic double-critical graphs :
\begin{prop}[Kawarabayashi, Pedersen, and Toft, 2010,~\cite{kpt10}]
Let $G \neq K_k$ be a double-critical $k$-chromatic graph, then
\begin{itemize}
\item The graph $G$ does not contains $K_{k-1}$ as a subgraph,
\item The graph $G$ has minimum degree at least $k + 1$,
\item For all edges $uv \in E(G)$ and all $(k-2)$-coloring of $G - u -
  v$, the set of common neighbors of $u$ and $v$ in $G$ contains
  vertices from every color class.
\end{itemize}
\label{prop:kpt}
\end{prop}

In particular, the last item implies that every edge belongs to at
least $k - 2$ triangles.

\begin{thm}
Every double-critical $k$-chromatic graph, for $k\le 8$, contains
$K_k$ as a minor.
\end{thm}
\begin{proof}
Consider for contradiction a $K_k$-minor free graph $G$ that is
double-critical $k$-chromatic. By the second item of
Proposition~\ref{prop:kpt}, $\delta(G)\ge k+1$. By
Theorem~\ref{th:krtri} and Theorem~\ref{th:k8-deg9}, this graph
has an edge that belongs to at most $k-3$ triangles. This
contradicts the last item of Proposition~\ref{prop:kpt}.
\end{proof}

Let us now give an alternative proof of the case $k=8$ that does not
need Theorem~\ref{th:k8-deg9}, but uses Theorem~\ref{th:k8tri}
instead. This might be usefull to prove the next case of
Conjecture~\ref{conj:dchadwiger}.

Consider for contradiction a $K_8$-minor free graph $G$ that is
double-critical $8$-chromatic.  By Theorem~\ref{th:k8tri} this graph
has an edge that belongs to at most $5$ triangles or contains
$K_{2,2,2,2,2}$ as an induced subgraph.  By Proposition~\ref{prop:kpt}
every edge of $G$ belongs to at least 6 triangles, thus $G$ contains
$K_{2,2,2,2,2}$ as an induced subgraph.  Let us denote $K\subseteq
V(G)$ the vertex set of a copy of $K_{2,2,2,2,2}$ in $G$. As
$K_{2,2,2,2,2}$ is maximal $K_8$-minor free, any connected component
$C$ of $G\setminus K$ is such that $N(C)\subset K$ induces a clique.
As $G$ is double-critical $8$-chromatic, there exists a $6$-coloring
of $G[N[C]]$, and a $6$-coloring of $G\setminus C$. As these two
graphs intersect on a clique one can combine their colorings and thus
obtain a 6-coloring of $G$, a contradiction.

\section{Application for coloration of $K_d$-minor free graphs}\label{sec:coloration}

Hadwiger's conjecture says that every $t$-chromatic graph $G$
(i.e. $\chi(G) =t$) contains $K_t$ has a minor. This conjecture has
been proved for $t \leq 6$, where the case $t = 5$ is equivalent to
the Four Color Theorem by Wagner's structure theorem of $K_5$-minor
free graph, and the case $t = 6$ has been proved by Robertson, Seymour
and Thomas~\cite{rst1}. The conjecture remains open for $t \geq 7$.
For $t = 7$ (resp. $t = 8$) the conjecture asks $K_7$-minor free graphs
(resp. $K_8$-minor free graphs) to be $6$-colorable (resp. $7$-colorable).
Using Claim~\ref{claim-alpha-Nv} and the $9$-degeneracy
(resp. $11$-degeneracy) of these graphs, one can prove that they are
$9$-colorable (resp. 11-colorable). We improve these bounds by one.

A graph $G$ is said to be \emph{$t$-minor-critical} if $\chi(G) = t$
and $\chi(H) < t$ whenever $H$ is a strict minor of $G$.  Hadwiger's
conjecture can thus be reformulated as follows : Every
$t$-minor-critical graph contains $K_t$ has a minor.  In the following
$\alpha(S)$ means $\alpha(G[S])$,the independence number of
$G[S]$. The following is a folklore claim, here for completeness.

\begin{claim}
\label{claim-alpha-Nv}
Given a $k$-minor critical graph $G$, for every vertex $v\in V(G)$ we
have that $\deg(v) + 2 - \alpha(N(v)) \ge k$.
\end{claim}

\begin{proof}
Given a vertex $v$ and a stable set $S$ of $N(v)$, consider the graph
$G'$ obtained from $G$ by contracting the edges between $v$ and
$S$. Since $G'$ is a strict minor of $G$ it is
$(k-1)$-colorable. Given such coloring of $G'$, one can $(k-1)$-color
$G\setminus\{v\}$ in such a way that all the vertices of $S$ have the
same color assigned. In this coloring at most $\deg(v) +1 - |S|$
colors are used in $N(v)$, thus $\deg(v) +2 - |S|$ colors
are sufficient to color $G$, and thus $\deg(v) + 2 - \alpha(N(v)) \ge k$.
\end{proof}

A \emph{split graph} is a graph which vertices can be partionned into
one set inducing a clique, and one set inducing an independent set.
These graphs are the graphs that do not contain $C_4$, $C_5$ or $2K_2$
as induced subgraphs~\cite{fh1}.

\begin{claim}
\label{claim-no-split}
Given a $k$-minor critical graph $G$, every separator $(A,B)$ of $G$
is such that $G[A \cap B]$ is not a split graph (i.e $G[A \cap B]$
contains $C_4$, $C_5$ or $2K_2$ as an induced subgraph).
\end{claim}

\begin{proof}
Assume by contradiction that there exists such separator $(A',B')$.
This implies the existence of a separator $(A,B)$ such that $S = A\cap
B \subseteq A'\cap B'$ , and such that each $G[A\setminus S]$ and
$G[B\setminus S]$ have a connected component, $C_{A}$ and $C_{B}$ such
that $N(C_{A})=N(C_{B})=S$. Note that $G[S]$ is a split graph and let
$I$ be one of its maximum independent sets and let $K=S \setminus I$
be a clique. Let $G_A$ and $G_B$ be the graphs respectively obtained
from $G[A]$ and $G[B]$ by identifying the vertices of $I$ into a
single vertex $i$. By maximality of $I$, in both graphs the vertex set
$K\cup \{i\}$ induces a clique.  Furthermore, these graphs are strict
minors of $G$ as the identification of the vertices in $I$ can be done
by contracting edges incident to $C_B$ or $C_A$ respectively.  Thus,
these graphs are $(k-1)$-colorable and these colorings imply the
existence of compatible $(k-1)$-colorings of $G[A]$ and $G[B]$, since
in both colorings the vertices of $I$ use the same color, and each
vertex of $K$ uses a distinct color. This yields in a $(k-1)$-coloring
of $G$, a contradiction.
\end{proof}

\begin{thm}
$K_7$-minor free graphs are $8$-colorable.
$K_8$-minor free graphs are $10$-colorable.
\end{thm}

\begin{proof}
Consider by contradiction that there is a $K_7$-minor free graph $G$
non-8-colorable (resp. a $K_8$-minor free graph $G$ non-10-colorable).
This graph is chosen such that $|E(G)|$ is minimal, this graph is
thus 9-minor-critical (resp. 11-minor-critical).

For any vertex $v$, since $\alpha(N(v))$ is at least 1,
Claim~\ref{claim-alpha-Nv} implies that $\deg(v) > 7$ (resp. $\deg(v)
> 9$).  If $\deg(v) = 8$ (resp. $\deg(v) = 10$), since $G$ is
$K_7$-minor free (resp. $K_8$-minor free), we have $\alpha(N(v)) \ge
2$, contradicting Claim~\ref{claim-alpha-Nv}. Finally if $\deg(v) = 9$
(resp. $\deg(v) = 11$), Claim~\ref{claim-alpha-Nv} implies that
$3>\alpha(N(v))$, and since $N(v)$ cannot be a clique, $\alpha(N(v)) =
2$. Thus with Mader's theorem we have that $\delta(G)= 9$
(resp. $\delta(G)= 11$), and that for every vertex $v$ of degree 9
(resp. of degree 11), $\alpha(N(v)) = 2$.  By Theorem~\ref{th:krtri}
(resp. Theorem~\ref{th:k8triweak}), we consider a vertex $u$ of degree
9 (resp. 11) such that there is an edge $uv$ which belongs to at most
$4$ (resp. $5$) triangles. Let $H = G[N(u)]$, and recall that
$\alpha(H) = 2$.

\begin{claim}
\label{claim-no-K5}
The graph $H = G[N(u)]$ does not contain a $K_5$ (resp. a $K_6$).
\end{claim}

\begin{proof}
Assume by contradiction that $H$ contains a $K_t$ with vertices
$x_1,\ldots,x_t$, for $t=5$ (resp. for $t=6$). Assume first that the
graph induced by $Y = N(u) \setminus \{x_1,\ldots,x_t\}$ is
connected. Since $\delta(G) \ge 9$ every vertex $x_i$ has a neighbor
in $Y$ or a neighbor $w_i$ in $G \setminus N[u]$. In the latter case,
denote $C_i$ the connected component of $w_i$ in $G\setminus N[u]$.
Since by Claim~\ref{claim-no-split} (for the partition $(N[C_i], V(G)
\setminus C_i)$) $N(C_i)$ intersects $Y$, one can contract $Y\cup
(V(G) \setminus N[u])$ into a single vertex and form a $K_{t+2}$
together with vertices $u,x_1,\ldots,x_t$, a contradiction.

Assume now that the graph induced by $Y$ is not connected and let
$y_1, y_2 \in Y$ be non-adjacent vertices. Since $G$ is $(2t-1)$-minor
critical, consider a $(2t-2)$-coloring of the graph $G'$ obtained from
$G$ by contracting $uy_1$ and $uy_2$. This coloring implies the
existence of a $(2t-2)$-coloring $c$ of $G\setminus u$ such that
$c(y_1) = c(y_2)$. As this coloring does not extends to $G$, the
$2t-1$ vertices in $N(u)$ use all the $(2t-2)$ colors. This implies
that the colors used for the $x_i$ are used only once in $N(u)$, and
that there exists a vertex $z \in Y$ which color is used only once in
$N(u)$. Assume $c(x_i)=i$ and $c(z)=7$. Given two colors $a$, $b$ and
a vertex $v$ colored $a$, the \emph{$(a,b)$-component of $v$} is the
the connected component of $v$ in the graph induced by $a$- or
$b$-colored vertices. For any $1\le i \le t$, suppose we switch colors
in the $(i,7)$-component of $z$. As this cannot lead to a coloring
which does not use all the colors in $N(u)$, there exists a
$(7,i)$-bicolored path from $z$ to $x_i$. This is impossible as
contracting these paths on $z$ would lead to a $K_{t+2}$ (with vertex
set $\{u,z,x_1,\ldots,x_t\}$). This concludes the proof of the claim.
\end{proof}

Let $v$ be a vertex of $H$ with minimum degree in $H$.  By the choice
of $u$ and Theorem~\ref{th:krtri} (resp. Theorem~\ref{th:k8triweak}),
$\deg_{H}(v)\le 4$ (resp. $\deg_{H}(v)\le 5$).

\begin{claim}
$\delta(H) = \deg_H(v) = 4$ (resp. $\delta(H) = \deg_H(v) = 5$).
\end{claim}
\begin{proof}
Since $\alpha(H)=2$, the non-neighbors of $v$ in $H$ form a
clique. Furthermore since $H$ does not contain a $K_5$ (resp. a $K_6$)
we have that $9 - 1 - \deg_{H}(v) < 5$ (resp. that $11 - 1 -
\deg_{H}(v) < 6$), and hence $\deg_{H}(v) = 4$ (resp. $\deg_{H}(v) =
5$).
\end{proof}

Let $y_1,\ldots,y_t$ with $t = 4$ (resp. $t = 5$) be the neighbors of
$v$ in $H$, and let $K$ be the $t$-clique formed by its non-neigbors.
By Claim~\ref{claim-no-K5} we can assume that $y_1$ and $y_2$ are
non-adjacent. Note that since $\alpha(G[N(u)]) = 2$ every vertex of
$K$ is adjacent to $y_1$ or $y_2$. Since $G$ is $(2t+1)$-minor
critical, consider a $2t$-coloring of the graph $G'$ obtained from $G$
by contracting $uy_1$ and $uy_2$. This coloring implies the existence
of a $2t$-coloring $c$ of $G \setminus u$ such that $c(y_1) = c(y_2)$. As
this coloring does not extends to $G$, the $2t + 1$ vertices in $N(u)$
use all the $2t$ colors. In particular, the colors used by $K$ (say
$1,\ldots t$) and $y_3$ (say 6) are thus used only once in $N(u)$.
For any $1\le i\le t$, suppose we switch colors in the $(i,6)$-component
of $y_3$. As this cannot lead to a coloring which does not use all the
colors in $N(u)$, there exists a $(i,6)$-bicolored path from $y_3$ to
the $i$-colored vertex of $K$. This is impossible as contracting these
paths on $y_3$, and contracting the edges $vy_1$ and $vy_2$ on $v$
would lead to a $K_{t+2}$ with vertex set $\{u,v,y_3\} \cup K$. This
concludes the proof of the theorem.
\end{proof}

\section{Conclusion}\label{sec:concl}

Theorem~\ref{thm:global} gives a sufficient condition for a graph to
have a $K_k$-minor. We wonder whether this condition is stronger than
Mader's Theorem : Is there a graph $G$ with a $K_k$-minor, for $4\le
k\le 7$, that has $m\le (k-2)n - {k-1 \choose 2}$ edges and $t\ge
m(k-3)/2$ triangles ?

We believe that our work can be extended to the next case.  Song and
Thomas~\cite{st1} proved a Mader-like theorem, similar to
Theorem~\ref{th:jorg} in the case of $K_9$-minor free graphs.

\begin{thm}[Song and Thomas, 2006,~\cite{st1}]
Every graph on $n \geq 9$ vertices and at least $7n - 27$ edges either
has a $K_9$-minor or is a $(K_{1,2,2,2,2,2}, 6)$-cockade or is isomorphic to
$K_{2,2,2,3,3}$.
\end{thm}

Note that $K_{2,2,2,3,3}$ has edges that belong to exactly $6$ triangles
and contains $K_{2,2,2,2,2,1}$ as a minor. We conjecture that we
can extend our main theorem as follows.
\begin{conj}
Let $G$ a graph such that every edge belongs to at least $7$ triangles
then either $G$ has a $K_9$-minor or contains $K_{1,2,2,2,2,2}$
as an induced subgraph.
\label{conj:extendk9}
\end{conj}

Proving this conjecture would have several consequences.
This would extend Theorem~\ref{thm:global}
as follows : Every graph $G$ with $m\ge 1$ edges and $t\ge 3m$ triangles
has a $K_9$ or $K_{1,2,2,2,2,2}$-minor.
It would also imply Conjecture~\ref{conj:munevo} for the case $k = 7$,
i.e.  $\mu(G)\le 7$ implies that $G$ is generically $7$-stress free.
Finally, it would imply Conjecture~\ref{conj:dchadwiger} for $k = 9$,
i.e. double-critical $9$-chromatic graphs have a $K_9$-minor.
We also conjecture that the following holds. In particular, it would
imply that $K_9$-minor free graphs are $12$-colorable (using the same
arguments as in Section \ref{sec:coloration}).
\begin{conj}
Any $K_9$-minor free graph $G$ with $\delta(G)=13$ has an edge $uv$
such that $u$ has degree $13$ and $uv$ belongs to at most $6$
triangles.
\label{conj:extendk9weak}
\end{conj}

We also believe that these structural properties on graph with edges
belonging to many triangles can actually be extended to
matroids. Graph minors can be studied in the more general context of
matroid minors~\cite{o1}. A triangle is then a circuit of size
$3$. Contrary to graphs, the case when every element of the matroid
belongs to $3$ triangles is already intricate. There are three
well-known matroids for which each element belongs to $3$ triangles :
the Fano matroid $F_7$, the uniform matroid $\mathcal{U}_{2,4}$, and
the graphical matroid $\mathcal{M}(K_5)$.  We conjecture that the
following holds.

\begin{conj}
Let $\mathcal{M}$ be a matroid where each element is contained in $3$
triangles, then $\mathcal{M}$ admits $\mathcal{M}(K_5)$, $F_7$ or
$\mathcal{U}_{2,4}$ as a minor.
\end{conj}

\bibliographystyle{plain}
\nocite{*}
\bibliography{biblio}
\end{document}